\documentclass[a4paper,12pt]{amsart}
\calclayout
\usepackage{amssymb}
\usepackage{amsmath}
\usepackage{amsfonts}
\usepackage[all]{xy}

\newtheorem{thm}{Theorem}[section]
\newtheorem{cor}[thm]{Corollary}

\newtheorem{prop}[thm]{Proposition}
\theoremstyle{definition}

\theoremstyle{remark}
\newtheorem{rem}[thm]{Remark}
\numberwithin{equation}{section}

\newcommand{\bb}[1]{\mathbb{#1}}
\newcommand{\cl}[1]{\mathcal{#1}}

\begin{document}

\title{An approximation theorem for nuclear operator systems}

\author{Kyung Hoon Han}
\address{Department of Mathematical Sciences, Seoul National University, San 56-1 ShinRimDong, KwanAk-Gu, Seoul
151-747, Republic of Korea} \email{kyunghoon.han@gmail.com}

\author[V.~I.~Paulsen]{Vern I.~Paulsen}
\address{Department of Mathematics, University of Houston,
Houston, TX 77204-3476, U.S.A.} \email{vern@math.uh.edu}



\date{}

\dedicatory{}

\commby{}


\begin{abstract}
We prove that an operator system $\mathcal S$ is nuclear in the
category of operator systems if and only if there exist nets of
unital completely positive maps $\varphi_\lambda : \cl S \to
M_{n_\lambda}$ and $\psi_\lambda : M_{n_\lambda} \to \cl S$ such
that $\psi_\lambda \circ \varphi_\lambda$ converges to ${\rm
id}_{\cl S}$ in the point-norm topology. Our proof is independent
of the Choi-Effros-Kirchberg characterization of nuclear
$C^*$-algebras and yields this characterization as a corollary. We
give an explicit example of a nuclear operator system that is not
completely order isomorphic to a unital $C^*$-algebra.
\end{abstract}

\maketitle

\section{Introduction}

In summary, we prove that an operator system $\cl S$ has the
property that for every operator system $\cl T$ the {\em  minimal
operator system tensor product}  $\cl S \otimes_{\min} \cl T$
coincides with the {\em maximal operator system tensor product}
$\cl S \otimes_{\max} \cl T$ if and only if there is a point-norm
factorization of $\cl S$ through matrices of the type described in
the abstract. Our proof of this fact is quite short, direct and
independent of the corresponding factorization results of Choi,
Effros and Kirchberg for nuclear $C^*$-algebras. Our proof uses in
a key way a characterization of the maximal operator system tensor
product given in \cite{KPTT1}. We are then able to deduce the
Choi-Effros-Kirchberg characterization of nuclear $C^*$-algebras
as an immediate corollary. The proof that one obtains in this way
of the Choi-Effros-Kirchberg result combines elements of the
proofs given in \cite{CE3} and \cite{Pi} but eliminates the need
to approximate maps into the second dual or to introduce
decomposable maps. Finally, we give a fairly simple example of an
operator system that is {\em nuclear} in this sense, but is not
completely order isomorphic to any $C^*$-algebra and yet has
second dual completely order isomorphic to $B(\ell^2(\bb N)).$
Earlier, Kirchberg and Wassermann\cite{KW} constructed a nuclear
operator system that is not even embeddable in any nuclear
$C^*$-algebra.

In \cite{Ka}, Kadison characterized the unital subspaces of a real
continuous function algebra on a compact set by observing that the
norm of a real continuous function algebra is determined by the
unit and the order. As for its noncommutative counterpart, Choi
and Effros gave an abstract characterization of the unital
involutive subspaces of $\cl B(\cl H)$ \cite{CE1}. The observation
that the unit and the matrix order in $\cl B(\cl H)$ determine the
matrix norm is key to their characterization. The former is called
a real function system or a real ordered vector space with an
Archimedean order unit while the latter is termed an operator
system.

Although the abstract characterization of an operator system played a key role in the work of Choi and Effros \cite{CE1} on the tensor products of $C^*$-algebras, there had not
been much attempt to study the categorical aspects of operator
systems and their tensor theory until a series of papers
\cite{PT, PTT, KPTT1, KPTT2}.
In particular,  \cite{KPTT1} introduced axioms for tensor products of operator systems and characterized the minimal and maximal tensor products of operator systems.

The positive cone of the minimal tensor product is the largest
among all possible positive cones of operator system tensor
products while that of the maximal tensor product is the smallest.
These extend the minimal tensor product and the maximal tensor
product of $C^*$-algebras. In other words, the minimal
(respectively, maximal) operator system tensor product of two
unital $C^*$-algebras is the operator subsystem of their minimal
(respectively, maximal) $C^*$-tensor product.

For the purposes of this paper, a unital $C^*$-algebra $\cl A$
will be called {\em $C^*$-nuclear} if and only if it has the
property that for every unital $C^*$-algebra $\cl B$ the minimal
$C^*$-tensor product $\cl A \otimes_{C^*\min} \cl B$ is equal to
the maximal $C^*$-tensor product $\cl A \otimes_{C^*\max} \cl B$.
We say that a $C^*$-algebra $\cl A$ has the {\em completely
positive approximation property} (in short, CPAP) if there exists
a net of unital completely positive maps $\varphi_\lambda : \cl A
\to \cl A$ with finite rank which converges to ${\rm id}_{\cl A}$
in the point-norm topology. The Choi-Effros-Kirchberg result is
that a $C^*$-algebra $\cl A$ is $C^*$-nuclear if and only if $\cl
A$ has the CPAP if and only if there exist nets of unital
completely positive maps $\varphi_\lambda : \cl A \to
M_{n_\lambda}$ and $\psi_\lambda : M_{n_\lambda} \to \cl A$ such
that $\psi_\lambda \circ \varphi_\lambda$ converges to ${\rm
id}_{\cl A}$ in the point-norm topology \cite{CE3, Ki1}. For a
recent proof which uses operator space methods and the
decomposable approximation, we refer the reader to \cite[Chapter
12]{Pi}.

An operator system will be called {\em nuclear} provided that the
minimal tensor product of it with an arbitrary operator system
coincides with the maximal tensor product. In \cite{KPTT1}, this
property was called $(\min,\max)$-nuclear. It is natural to ask
whether the approximation theorems of nuclear $C^*$-algebras
\cite{CE3, Ki1} also hold in the category of operator systems. In
section 3, we show that an operator system $\cl S$ is nuclear if
and only if there exist nets of unital completely positive maps
$\varphi_\lambda : \cl S \to M_{n_\lambda}$ and $\psi_\lambda :
M_{n_\lambda} \to \cl S$ such that $\psi_\lambda \circ
\varphi_\lambda$ converges to ${\rm id}_{\cl S}$ in the point-norm
topology.

We then prove, independent of the Choi-Effros-Kirchberg theorem,
that a $C^*$-algebra is $C^*$-nuclear if and only if it is nuclear
as an operator system. Thus, we obtain the Choi-Effros-Kirchberg
characterization as a corollary of the factorization result for
operator systems.

In contrast, CPAP does not imply nuclearity in the category of
operator systems. Let
$$\cl S_0 = {\rm span} \{ E_{1,1}, E_{1,2}, E_{2,1}, E_{2,2}, E_{2,3},
E_{3,2}, E_{3,3} \} \subset M_3.$$ In \cite[Theorem~5.18]{KPTT1},
it is shown that this finite dimensional operator system $\cl S_0$
is not nuclear.

On the other hand, \cite[Theorem~5.16]{KPTT1} shows that the
minimal and maximal operator system tensor products of $\cl S_0
\otimes \cl B$ coincide for every unital $C^*$-algebra $\cl B.$
Thus, for operator systems, tensoring with $C^*$-algebras is not
sufficient to discern ordinary nuclearity, i.e.,
$(\min,\max)$-nuclearity. However, it is easily seen that the
minimal and maximal operator system tensor products of $\cl S
\otimes \cl B$ coincide for every unital $C^*$-algebra $\cl B$ if
and only if $\cl S$ is $(\min,{\rm c})$-nuclear, in the sense of
\cite{KPTT1}.

Finally, in section 4, we construct a nuclear operator system that
is not unitally, completely order isomorphic to a unital
$C^*$-algebra. This shows that the theory of nuclear operator
systems properly extends the theory of nuclear $C^*$-algebras. In
contrast, by \cite{CE1}, every injective operator system is
unitally, completely order isomorphic to a unital $C^*$-algebra.

\section{preliminaries}

Let $\cl S$ and $\cl T$ be operator systems. Following
\cite{KPTT1}, an {\em operator system structure} on $\cl S \otimes
\cl T$ is defined as a family of cones $M_n (\cl S \otimes_\tau
\cl T)^+$ satisfying:
\begin{enumerate}
\item[(T1)] $(\cl S \otimes \cl T, \{ M_n (\cl S \otimes_\tau \cl
T)^+ \}_{n=1}^\infty, 1_{\cl S} \otimes1_{\cl T})$ is an operator
system denoted by $\cl S \otimes_\tau \cl T$, \item[(T2)] $M_n(\cl
S)^+ \otimes M_m(\cl T)^+ \subset M_{mn} (\cl S \otimes_\tau \cl
T)^+$ for all $n,m \in \mathbb N$, and \item[(T3)] if $\varphi :
\cl S \to M_n$ and $\psi : \cl T \to M_m$ are unital completely
positive maps, then $\varphi \otimes \psi : \cl S \otimes_\tau \cl
T \to M_{mn}$ is a unital completely positive map.
\end{enumerate}
By an {\em operator system tensor product,} we mean a mapping
$\tau : \cl O \times \cl O \to \cl O$, such that for every pair of
operator systems $\cl S$ and $\cl T$, $\tau (\cl S, \cl T)$ is an
operator system structure on $\cl S \otimes \cl T$, denoted $\cl S
\otimes_\tau \cl T$. We call an operator system tensor product
$\tau$ {\em functorial,} if the following property is satisfied:
\begin{enumerate}
\item[(T4)] For any operator systems $\cl S_1, \cl S_2, \cl T_1,
\cl T_2$ and unital completely positive maps $\varphi : \cl S_1
\to \cl T_1, \psi : \cl S_2 \to \cl T_2$, the map $\varphi \otimes
\psi : \cl S_1 \otimes \cl S_2 \to \cl T_1 \otimes \cl T_2$ is
unital completely positive.
\end{enumerate}
An operator system structure is defined on two fixed operator
systems, while the functorial operator system tensor product can
be thought of as the bifunctor on the category consisting of
operator systems and unital completely positive maps.

Given an operator system $\cl R$ we let $S_n(\cl R)$ denote the set of
unital completely positive maps of $\cl R$ into $M_n$.
For operator systems $\cl S$ and $\cl T$, we put
$$M_n(\cl S \otimes_{\min} \cl T)^+ = \{ [p_{i,j}]_{i,j} \in
M_n(\cl S \otimes \cl T) : \forall \varphi \in S_k(\cl S), \psi
\in S_m(\cl T), [(\varphi \otimes \psi)(p_{i,j})]_{i,j} \in
M_{nkm}^+ \}.$$  Then the family $\{ M_n(\cl S
\otimes_{\min} \cl T)^+ \}_{n=1}^\infty$ is an operator system
structure on $\cl S \otimes \cl T.$
Moreover, if we let $\iota_{\cl S} : \cl S \to \cl B(\cl H)$
and $\iota_{\cl T} : \cl T \to \cl B(\cl K)$ be any unital completely
order isomorphic embeddings, then it is shown in \cite{KPTT1} that this is the operator system structure on
$\cl S \otimes \cl T$
 arising from the embedding
$\iota_{\cl S} \otimes \iota_{\cl T} : \cl S \otimes \cl T \to \cl
B(\cl H \otimes \cl K)$. As in \cite{KPTT1}, we call the operator
system $(\cl S \otimes \cl T, \{ M_n(\cl S \otimes_{\min} \cl T)
\}_{n=1}^\infty, 1_{\cl S} \otimes 1_{\cl T})$ the {\em minimal}
tensor product of $\cl S$ and $\cl T$ and denote it by $\cl S
\otimes_{\min} \cl T$.

The mapping $\min : \cl O \times \cl O \to \cl O$ sending $(\cl S,
\cl T)$ to $\cl S \otimes_{\min} \cl T$ is an injective,
associative, symmetric and functorial operator system tensor
product. The positive cone of the minimal tensor product is the
largest among all possible positive cones of operator system
tensor products \cite[Theorem~4.6]{KPTT1}. For $C^*$-algebras $\cl
A$ and $\cl B$, we have the completely order isomorphic inclusion
$$\cl A \otimes_{\min} \cl B \subset \cl A
\otimes_{\rm C^*\min} \cl B$$ \cite[Corollary~4.10]{KPTT1}.

For operator systems $\cl S$ and $\cl T$, we put
$$D_n^{\max}(\cl S, \cl T) = \{ \alpha(P \otimes Q)
\alpha^* : P \in M_k(\cl S)^+, Q \in M_l(\cl T)^+, \alpha \in
M_{n,kl},\ k,l \in \mathbb N \}.$$ Then it is a matrix ordering on
$\cl S \otimes \cl T$ with order unit $1_{\cl S} \otimes 1_{\cl
T}$. Let $\{ M_n(\cl S \otimes_{\max} \cl T)^+ \}_{n=1}^\infty$ be
the Archimedeanization of the matrix ordering $\{ D_n^{\max}(\cl
S, \cl T) \}_{n=1}^\infty$. Then it can be written as
$$M_n(\cl S \otimes_{\max} \cl T)^+ = \{ X \in M_n(\cl S
\otimes \cl T) : \forall \varepsilon>0, X+\varepsilon I_n \otimes
1_{\cl S} \otimes 1_{\cl T} \in D_n^{\max}(\cl S, \cl T) \}.$$ We
call the operator system $(\cl S \otimes \cl T, \{ M_n(\cl S
\otimes_{\max} \cl T)^+ \}_{n=1}^\infty, 1_{\cl S} \otimes 1_{\cl
T})$ the {\em maximal} operator system tensor product of $\cl S$
and $\cl T$ and denote it by $\cl S \otimes_{\max} \cl T$.

The mapping $\max : \cl O \times \cl O \to \cl O$ sending $(\cl S,
\cl T)$ to $\cl S \otimes_{\max} \cl T$ is an associative,
symmetric and functorial operator system tensor product. The
positive cone of the maximal tensor product is the smallest among
all possible positive cones of operator system tensor products
\cite[Theorem~5.5]{KPTT1}. For $C^*$-algebras $\cl A$ and $\cl B$,
we have the completely order isomorphic inclusion
$$\cl A \otimes_{\max} \cl B \subset \cl A
\otimes_{\rm C^*\max} \cl B$$ \cite[Theorem~5.12]{KPTT1}.

\section{An approximation theorem for nuclear operator systems}

We prove the main theorem of this paper which generalizes the
Choi-Effros-Kirchberg approximation theorem. The proof is quite
simple compared to the original one. In particular, the proof does
not depend on the Kaplansky density theorem.

\begin{thm}\label{main1}
Suppose that $\Phi : \cl S \to \cl T$ is a unital completely
positive map for operator systems $\cl S$ and $\cl T$. The
following are equivalent:
\begin{enumerate}
\item[(i)] the map $${\rm id}_{\cl R} \otimes \Phi : \cl R
\otimes_{\min} \cl S \to \cl R \otimes_{\max} \cl T$$ is
completely positive for any operator system $\cl R$; \item[(ii)]
the map $${\rm id}_E \otimes \Phi : E \otimes_{\min} \cl S \to E
\otimes_{\max} \cl T$$ is completely positive for any finite
dimensional operator system $E$; \item[(iii)] there exist nets of
unital completely positive maps $\varphi_\lambda : \cl S \to
M_{n_\lambda}$ and $\psi_\lambda : M_{n_\lambda} \to \cl T$ such
that $\psi_\lambda \circ \varphi_\lambda$ converges to the map
$\Phi$ in the point-norm topology.
$$\xymatrix{\cl S \ar[rr]^\Phi \ar[dr]_{\varphi_\lambda} && \cl T \\
& M_{n_\lambda} \ar[ur]_{\psi_\lambda} &}$$
\end{enumerate}
\end{thm}

\begin{proof}
Clearly, (i) implies (ii).

$\rm (iii)\Rightarrow (i).$  For any operator system $\cl R$ and any
$n \in \bb N,$ if we identify $M_k(M_n \otimes \cl R) = M_{nk} \otimes
\cl R$ in the usual manner, then a somewhat tedious calculation shows
that $D_k^{\max}(M_n, \cl R) = M_{nk}(\cl R)^+ = M_k(M_n
\otimes_{\min} \cl R)^+.$ This gives an independent verification that $\cl R \otimes_{\max} M_n = \cl R \otimes_{\min} M_n,$ i.e., that
the two operator system structures are identical. Alternatively, this
fact follows from  \cite[Corollary~6.8]{KPTT1}, which they point out
is obtained independently of the Choi-Effros-Kirchberg theorem.  From the maps
$$\xymatrix{\cl R \otimes_{\min} \cl S \ar[rr]^-{{\rm id}_{\cl R} \otimes \varphi_\lambda}
 && \cl R \otimes_{\min} M_{n_\lambda} = \cl R \otimes_{\max} M_{n_\lambda}
 \ar[rr]^-{{\rm id}_{\cl R} \otimes \psi_\lambda} && \cl R \otimes_{\max} \cl T,}$$
we see that the map $${\rm id}_{\cl R} \otimes \psi_\lambda \circ
\varphi_\lambda : \cl R \otimes_{\min} \cl S \to \cl R
\otimes_{\max} \cl T$$ is completely positive for any operator
system $\cl R$. Since $\|\cdot \|_{\cl R \otimes_{\max} \cl T}$ is
a cross norm, ${\rm id}_{\cl R} \otimes (\psi_\lambda \circ
\phi_\lambda) (z)$ converges to ${\rm id}_{\cl R} \otimes \Phi(z)$
for each $z \in \cl R \otimes \cl S$. It follows that $z \in (\cl
R \otimes_{\min} \cl S)^+$ implies ${\rm id}_{\cl R} \otimes \Phi
(z) \in (\cl R \otimes_{\max} \cl T)^+$.

\vskip 1pc

$\rm (ii)\Rightarrow (iii).$ Let $E$ be a finite dimensional
operator subsystem of $\cl S$. There exists a state $\omega_1$ on
$E$ which plays a role of the non-canonical Archimedean order unit
on the dual space $E^*$ \cite[Corollary~4.5]{CE1}. In other words,
$(E^*, \omega_1$) is an operator system. We can regard the
inclusion $\iota : E \subset \cl S$ as an element in $(E^*
\otimes_{\min} \cl S)^+$ \cite[Lemma~8.4]{KPTT2}. The restriction
$\Phi |_E : E \to \cl T$ can be identified with the element $({\rm
id}_{E^*} \otimes \Phi) (\iota)$. By assumption, it belongs to
$(E^* \otimes_{\max} \cl T)^+$. We consider the directed set
$$\Omega = \{(E, \varepsilon) : \text{$E$ is a finite dimensional
operator subsystem of $\cl S$}, \varepsilon >0 \}$$ with the
standard partial order. Let $\lambda = (E, \varepsilon)$. For any
$\varepsilon >0$, the restriction $\Phi|_E$ can be written as
$$\Phi|_E + \varepsilon \omega_1 \otimes 1_{\cl T} = \alpha f
\otimes Q \alpha^*$$ for $\alpha \in M_{1,n_{\lambda}m}, f \in
M_{n_\lambda}(E^*)^+$ and $Q \in M_m(\cl T)^+$. The map $f : E \to
M_{n_\lambda}$ is completely positive and the matrix $f(1_{\cl
S})$ is positive semi-definite. Let $P$ be the support projection
of $f(1_{\cl S})$. For $x \in \cl S^+$, we have $$0 \le f(x) \le
\|x\|f(1_{\cl S}) \le \|x\| \|f(1_{\cl S})\| P.$$ Since every
element in $\cl S$ can be written as a linear combination of
positive elements in $\cl S$, the range of $f$ is contained in $P
M_{n_\lambda} P$. The positive semi-definite matrix $f(1_{\cl S})$
is invertible in $P M_{n_\lambda} P$. We denote by $f(1_{\cl
S})^{-1}$ its inverse in $P M_{n_\lambda} P$. Put $p = {\rm rank}
P$ and let $U^* P U = I_{p} \oplus 0$ be the diagonalization of
$P$. Since we can write
$$\begin{aligned} & \alpha f \otimes Q \alpha^* \\ = & \alpha
(f(1_{\cl S})^{1 \over 2} U
\begin{pmatrix} I_p \\ 0 \end{pmatrix} \otimes I_m)
\cdot [\begin{pmatrix} I_p & 0 \end{pmatrix} U^* f(1_{\cl S})^{-{1
\over 2}}\ f\ f(1_{\cl S})^{-{1 \over 2}} U
\begin{pmatrix} I_p \\ 0
\end{pmatrix} \otimes Q] \ \\ & \cdot (f(1_{\cl S})^{1 \over 2} U
\begin{pmatrix} I_{ p} \\ 0 \end{pmatrix} \otimes I_m)^* \alpha^*,
\end{aligned}$$
we may assume that $f : E \to M_{n_\lambda}$ is a unital
completely positive map. By the Arveson extension theorem, $f : E
\to M_{n_\lambda}$ extends to a unital completely positive map
$\varphi_\lambda : \cl S \to M_{n_\lambda}$. We define a
completely positive map $\psi'_\lambda : M_{n_\lambda} \to \cl T$
by $$\psi'_\lambda(A) = \alpha A \otimes Q \alpha^*, \qquad A \in
M_{n_\lambda}.$$ For $x \in E$, we have
$$\|\Phi(x)-\psi'_\lambda \circ \varphi_\lambda (x) \| = \|\Phi(x)-\alpha
f(x) \otimes Q \alpha^*\| = \varepsilon \| \omega_1(x) 1_{\cl T}
\| \le \varepsilon \|x\|.$$ Hence, we can take nets of unital
completely positive maps $\varphi_\lambda : \cl S \to
M_{n_\lambda}$ and completely positive maps $\psi'_\lambda :
M_{n_\lambda} \to \cl T$ such that $\psi'_\lambda \circ
\varphi_\lambda$ converges to the map $\Phi$ in the point-norm
topology. Since each $\varphi_\lambda$ is unital, $\psi'_\lambda
(I_{n_\lambda})$ converges to $1_{\cl T}$. Let us choose a state
$\omega_\lambda$ on $M_{n_\lambda}$ and set
$$\psi_\lambda(A) = {1 \over \| \psi'_\lambda\|} \psi'_\lambda(A) +
\omega_\lambda(A) (1_{\cl T} - {1 \over \| \psi'_\lambda \|}
\psi'_\lambda (I_{n_{\lambda}})).$$ Then $\psi_\lambda :
M_{n_\lambda} \to \cl T$ is a unital completely positive map such
that $\psi_\lambda \circ \varphi_\lambda$ converges to the map
$\Phi$ in the point-norm topology.
\end{proof}

Putting $\cl S = \cl T$ and $\Phi = {\rm id}_{\cl S}$, we obtain
the following corollary.

\begin{cor}\label{main2}
Let $\cl S$ be an operator system. The following are equivalent:
\begin{enumerate}
\item[(i)] $\cl S$ is nuclear; \item[(ii)] we have $$E
\otimes_{\min} \cl S = E \otimes_{\max} \cl S$$ for any finite
dimensional operator system $E$; \item[(iii)] there exist nets of
unital completely positive maps $\varphi_\lambda : \cl S \to
M_{n_\lambda}$ and $\psi_\lambda : M_{n_\lambda} \to \cl S$ such
that $\psi_\lambda \circ \varphi_\lambda$ converges to ${\rm
id}_{\cl S}$ in the point-norm topology.
\end{enumerate}
\end{cor}

\begin{cor}[Choi-Effros-Kirchberg Theorem] Let $\cl A$ be a unital $C^*$-algebra.
Then $\cl A$ is $C^*$-nuclear if and only if there exist nets of
unital completely positive maps $\varphi_\lambda : \cl A \to
M_{n_\lambda}$ and $\psi_\lambda : M_{n_\lambda} \to \cl A$ such
that $\psi_\lambda \circ \varphi_\lambda$ converges to ${\rm
id}_{\cl A}$ in the point-norm topology.
\end{cor}
\begin{proof}
It will be enough to prove that if $\cl A$ is $C^*$-nuclear, then
for every operator system $\cl T,$ the minimal and maximal
operator system tensor products coincide on $\cl A \otimes \cl T.$
Again this fact follows from \cite[Corollary~6.8]{KPTT1} which is
independent of the Choi-Effros-Kirchberg theorem.

Since the notation is somewhat different in \cite{KPTT1} and their
result relies on several earlier results, we repeat the argument
below.

Let $C^*_u(\cl T)$ be the universal $C^*$-algebra generated by the
operator system $\cl T$ as defined in \cite{KPTT1}. Since $\cl A$
is $C^*$-nuclear, we have that $\cl A \otimes_{C^*\min} C^*_u(\cl
T) = \cl A \otimes_{C^*\max} C^*_u(\cl T).$ But we have that $\cl
A \otimes_{\min} \cl T \subseteq \cl A \otimes_{C^*\min} C^*_u(\cl
T)$ completely order isomorphically, by
\cite[Corollary~4.10]{KPTT1}. Also, by \cite[Theorem~6.4]{KPTT1}
the inclusion of the {\em commuting} tensor product $\cl A
\otimes_{\rm c} \cl T \subseteq \cl A \otimes_{C^*\max} C^*_u(\cl
T)$ is a complete order isomorphism.

Thus, the fact that $\cl A$ is $C^*$-nuclear implies that $\cl A
\otimes_{\min} \cl T = \cl A \otimes_{\rm c} \cl T$ completely
order isomorphically. Finally, the result follows from the fact
\cite[Theorem~6.7]{KPTT1}, that for any $C^*$-algebra $\cl A$,
$\cl A \otimes_{\rm c} \cl T = \cl A \otimes_{\max} \cl T,$
completely order isomorphically.
\end{proof}

\begin{rem} Suppose that we call an operator system $\cl S$ {\it $C^*$-nuclear} if $\cl S \otimes_{\min} \cl B = \cl S \otimes_{\max} \cl B$ for every unital $C^*$-algebra $\cl B.$
Then it follows by \cite[Theorem~6.4]{KPTT1}, that an operator
system $\cl S$ is $C^*$-nuclear if and only if $\cl S
\otimes_{\min} \cl T = \cl S \otimes_{\rm c} \cl T$ for every
operator system $\cl T$. In the terminology of \cite{KPTT1}, this
latter property is the definition of $(\min,{\rm c})$-nuclearity.
Thus, an operator system is $C^*$-nuclear if and only if it is
$(\min,{\rm c})$-nuclear. A complete characterization of such
operator systems is still unknown.
\end{rem}

By a result of Choi and Effros \cite{CE2}, a $C^*$-algebra $\cl A$
is nuclear if and only if its enveloping von Neumann algebra ${\cl
A}^{**}$ is injective. We wish to extend this result to nuclear
operator systems. In the next section we produce an example of a
nuclear operator system that is not completely order isomorphic to
any $C^*$-algebra.

An operator space $X$ is called {\it nuclear} provided that there
exist nets of complete contractions $\varphi_{\lambda}:X \to
M_{n_{\lambda}}$ and $\psi_{\lambda}: M_{n_{\lambda}} \to X$ such
that $\psi_\lambda \circ \varphi_\lambda$ converges to ${\rm
id}_X$ in the point-norm topology. Kirchberg \cite{Ki2} gives an
example of an operator space $X$ that is not nuclear, but such
that the bidual $X^{**}$ is completely isometric to an injective
von~Neumann algebra. A later theorem of Effros, Ozawa and
Ruan\cite[Theorem~4.5]{EOR} implies that Kirchberg's operator
space $X$ is also not locally reflexive. See \cite{ER} for further
details on local reflexivity.

These pathologies do not occur for operator systems. This follows
from the works of Kirchberg \cite{Ki2} and of Effros, Ozawa and
Ruan \cite{EOR}. The following summarizes their results.

\begin{thm}\label{main3}
Let $\cl S$ be an operator system. Then the following are
  equivalent:
\begin{enumerate}
\item[(i)] $\cl S$ is a nuclear operator system; \item[(ii)] $\cl
S$ is a nuclear operator space; \item[(iii)] $\cl S^{**}$ is
unitally completely order isomorphic to an injective
  von~Neumann algebra.
\end{enumerate}
\end{thm}
\begin{proof} Clearly, (i) implies (ii) by Theorem~\ref{main1}.

For (ii) $\Rightarrow$ (iii), combine \cite[Theorem~4.5]{EOR} and
\cite[Theorem~3.1]{CE1} and Sakai's theorem.

Finally, the proof that (iii) implies (i), is due to Kirchberg
\cite[Lemma~2.8(ii)]{Ki2}.
\end{proof}

Smith's characterization of nuclear $C^*$-algebras
\cite[Theorem~1.1]{Sm} follows from (ii) $\Rightarrow$ (i). We now
see another contrast between operator spaces and operator systems.

\begin{cor}
Let $\cl S$ be an operator system. If $\cl S^{**}$ is
unitally completely order isomorphic to an injective von~Neumann
algebra, then $\cl S$ is a locally reflexive operator space.
\end{cor}
\begin{proof} By the above result, $\cl S$ is a nuclear operator space and hence by
\cite[Theorem~4.4]{EOR}, $\cl S$ is locally reflexive.
\end{proof}

\begin{cor}
Every finite dimensional nuclear operator system is unitally
completely order isomorphic to the direct sum of matrix algebras.
\end{cor}

\begin{proof} Let $\cl S$ be a finite dimensional operator
  system. Then $\cl S = \cl S^{**},$ which by the above result is
  unitally completely order isomorphic to a finite dimensional
  $C^*$-algebra.
\end{proof}

\begin{rem} Kirchberg \cite[Theorem~1.1]{Ki2} proves that every nuclear
separable operator system is unitally completely isometric to a
quotient of the CAR-algebra by a hereditary C*-subalgebra and that
conversely, every such quotient gives rise to a nuclear separable
operator system.
\end{rem}

\section{A Nuclear Operator system that is not a $C^*$-algebra}

Kirchberg and Wassermann\cite{KW} constructed a remarkable example of a nuclear
operator system that has no unital complete order embedding into any
nuclear $C^*$-algebra. So, in particular, they give an example of a
nuclear operator system that is not unitally completely order
isomorphic to a $C^*$-algebra. In this section we provide a very
concrete example of this latter phenomena.

Let $\cl K_0 \subseteq \cl B(\ell^2(\bb N))$ denote the norm
closed linear span of  $\{ E_{i,j}: (i,j) \ne (1,1) \},$ where
$E_{i,j}$ are the standard matrix units and let $$\cl S_0= \{
\lambda I + K_0 : \lambda \in \bb C, K_0 \in \cl K_0 \}\subseteq
\cl B(\ell^2(\bb N))$$ denote the operator system spanned by $\cl
K_0$ and the identity operator. The goals of this section are to
show that $\cl S_0$ is a nuclear operator system that it is not
unitally completely order isomorphic to any $C^*$-algebra and that
$\cl S_0^{**}$ is unitally completely order isomorphic to $\cl
B(\ell^2(\bb N)).$

Let $V_n: \bb C^n \to \ell^2(\bb N)$ be the isometric inclusion
defined by $V_n(e_j) =e_j, 1 \le j \le n$ and let $Q_n \in \cl
B(\ell^2(\bb N))$ be the projection onto the orthocomplement of
$V_n(\bb C^n).$ Finally, define unital completely positive maps,
$\varphi_n : \cl B(\ell^2(\bb N)) \to M_n$ and $\psi_n:M_n \to \cl
B(\ell^2(\bb N))$ by $$\varphi_n(X) = V_n^*XV_n \qquad \text{and}
\qquad \psi_n(Y) = V_nYV_n^* + y_{1,1}Q_n,\quad Y=(y_{i,j}).$$

\begin{prop} The following hold:
\begin{enumerate}
\item[(i)] $\psi_n(M_n) \subseteq \cl S_0;$ \item[(ii)] for any $m
\in \bb N$ and $(X_{i,j}) \in M_m(\cl S_0),$ $\|(X_{i,j}) -
(\psi_n \circ \varphi_n(X_{i,j})) \| \to 0$ as $n \to +\infty;$
\item[(iii)] $\cl S_0$ is a nuclear operator system.
\end{enumerate}
\end{prop}
\begin{proof} Given $Y \in M_n,$ we have that $\psi_n(Y - y_{1,1}I_n) \in \cl K_0,$ and hence $\psi_n(Y) \in \cl S_0$ and (i) follows.

If $X \in \cl K_0,$ then the first $n \times n$ matrix entries of
$\psi_n \circ \varphi_n(X)$ agree with those of $X$ and the
remaining entries are $0.$ Since $X$ is compact, $\|X - \psi_n
\circ \varphi_n(X) \| \to 0$ and since both maps are unital, we
have that (ii) holds for the case $m=1.$  The case $m >1$ follows
similarly.

Statement (iii) follows by (ii) and Theorem~\ref{main1}.
\end{proof}

\begin{thm} The nuclear operator system $\cl S_0$ is not unitally completely order isomorphic to a $C^*$-algebra.
\end{thm}
\begin{proof}
Assume to the contrary that $\cl A$ is a unital $C^*$-algebra and
that $\gamma: \cl A \to \cl S_0$ is a unital, complete order
isomorphism. Then $\gamma$ is also a completely isometric
isomorphism. Use the Stinespring representation \cite[Theorem
4.1]{Pa} to write $\gamma(a) = P\pi(a)P,$ where $\pi:\cl A \to \cl
B(\ell^2(\bb N) \oplus \cl H)$ is a unital $*$-homomorphism and
$P: \ell^2(\bb N) \oplus \cl H \to \ell^2(\bb N)$ denotes the
orthogonal projection.

Let $a_{i,j}, (i,j) \ne (1,1)$ denote the unique elements of $\cl
A,$ satisfying $\gamma(a_{i,j}) = E_{i,j}.$  Relative to the
decomposition $\ell^2(\bb N) \oplus \cl H,$ we have that
\[ \pi(a_{i,j}) = \begin{pmatrix} E_{i,j} & B_{i,j}\\C_{i,j} &
  D_{i,j} \end{pmatrix}, \] where $B_{i,j}: \cl H \to
\ell^2(\bb N), C_{i,j}: \ell^2(\bb N) \to \cl H$ and $D_{i,j}: \cl H
\to \cl H$ are bounded operators.

By choosing an orthonormal basis $\{ u_t \}_{t \in T}$ we may
regard $B_{i,j}$ as an $\bb N \times T$ matrix and $C_{i,j}$ as a
$T \times \bb N$ matrix. Since $\|\pi(a_{i,j})\| = \|E_{i,j}\|
=1,$ we must have that the $i$-th row of $B_{i,j}$ is $0$ and the
$j$-th column of $C_{i,j}$ is $0.$

If $k \ne i,$ then
\[ 1= \|( E_{i,j}, E_{k,k+1})\| = \|(\pi(a_{i,j}), \pi(a_{k,k+1})) \| \ge
\|(E_{i,j}, B_{i,j}, E_{k,k+1}, B_{k,k+1})\| \]
from which it follows that the $k$-th row of $B_{i,j}$ is also $0.$
This proves that $B_{i,j} =0$ for all $(i,j) \ne (1,1).$


A similar argument using the fact that $\|\begin{pmatrix} E_{i,j} \\ E_{k+1,k} \end{pmatrix} \| =1$ for $k \ne j$ yields that $C_{i,j} =0$ for all $(i,j) \ne (1,1).$

Since $\cl A$ is the closed linear span of $a_{i,j}, (i,j) \ne (1,1)$ and the identity it follows that for any $a \in \cl A,$
\[ \pi(a) = \begin{pmatrix} \gamma(a) & 0 \\ 0 & \rho(a) \end{pmatrix}, \]
for some linear map $\rho: \cl A \to \cl B(\cl H).$

But since $\pi$ is a unital $*$-homomorphism, it follows that
$\gamma: \cl A \to \cl B(\ell^2(\bb N))$ is a unital
$*$-homomorphism and, consequently, that $\cl S_0$ is a
$C^*$-subalgebra of $\cl B(\ell^2(\bb N)).$ But $E_{1,2}, E_{2,1}
\in \cl S_0,$ while $E_{1,1} = E_{1,2}E_{2,1} \notin \cl S_0.$
This contradiction completes the proof.
\end{proof}

By Theorem \ref{main3}, we know that $\cl S_0^{**}$ is an
injective von~Neumann algebra, so it is interesting to identify
the precise algebra.

\begin{thm} $\cl S_0^{**}$ is unitally completely order isomorphic to $\cl B(\ell^2(\bb N)).$
\end{thm}
\begin{proof} We only prove that $\cl S_0^{**}$ is unitally order
  isomorphic to $\cl B(\ell^2(\bb N)).$ To this end, let $\cl S = \{
  \lambda I + K: \lambda \in \bb C, K \in \cl K(\ell^2(\bb N)) \},$ denote the
  unital $C^*$-algebra spanned by the compact operators $\cl
  K(\ell^2(\bb N))$ and the identity. Thus, $\cl S_0 \subseteq \cl S$
  is a codimension 1 subspace.

As vector spaces, we have that $\cl S = \bb C \oplus \cl
K(\ell^2(\bb N)),$ so that $\cl S^* = \bb C \oplus \cl
T(\ell^2(\bb N)),$ where this latter space denotes the trace class
operators.

We let $\delta_{i,j}: \cl K(\ell^2(\bb N)) \to \bb C$ denote the
linear functional satisfying
\[\delta_{i,j}(E_{k,l}) = \begin{cases} 1& i=k, j=l\\
0& \text{ otherwise} \end{cases} \] so that every element of $\cl
K(\ell^2(\bb N))^*$ is of the form $\sum_{i,j} t_{i,j} \delta_{i,j}$
for some trace class matrix $T= (t_{i,j}).$ We identify $\cl S^* =
\bb C \oplus \cl T(\ell^2(\bb N))$ where
\[ \langle (\beta, T), \lambda I + K \rangle = \beta \lambda +
\sum_{i,j} t_{i,j} k_{i,j} = \beta \lambda + {\rm tr}(T^tK) \]
with $K=(k_{i,j})$.

The functional $(\beta, T)$ is positive if and only if $T$ is a
positive operator and $\beta \ge {\rm tr}(T).$ If $(\beta, T)$ is
a positive functional on $\cl S$, then we have
$$0 \le \langle (\beta,T), K \rangle = {\rm tr}(T^t K) \quad \text{and}
\quad 0 \le \langle (\beta,T), I-I_n \rangle = \beta - {\rm
tr}(T^t I_n)$$ for all positive compact operators $K$ and $n \in
\bb N$. Let $\lambda I + K$ be a positive operator. Since $K$ is
compact, we have $\lambda \ge 0$. The converse follows from
$$\langle (\beta, T), \lambda I + K \rangle = \beta \lambda + {\rm
tr}(T^t K) \ge {\rm tr}(T^t (\lambda I +K)) \ge 0.$$

Identify $\cl S_0^*$ with $\bb C \oplus \cl T_0$ where $\cl T_0$
denotes the trace class operators $T_0=(t_{i,j})$ with $t_{1,1}
=0.$ Since every positive functional on $\cl S_0$ extends to a
positive functional on $\cl S$ by the Krein theorem, we have that
$(\beta, T_0)$ defines a positive functional if and only if there
exists $\alpha \in \bb C,$ such that $T= T_0 + \alpha E_{1,1}$ is
positive and $\beta \ge {\rm tr}(T_0) + \alpha.$ That is if and
only if $\beta \ge {\rm tr}(T),$ where $T$ is some positive trace
class operator equal to $T_0$ modulo the span of $E_{1,1}.$

In a similar fashion we may identify $\cl S_0^{**}$ as the vector space
 $\bb C \oplus \cl B_0$, where $X_0=(x_{i,j}) \in \cl B_0$ if
and only if $X_0$ is bounded and $x_{1,1} =0.$ Moreover, $(\mu, X_0)$ will
define a positive element of $\cl S_0^{**}$ if and only if
\[ \mu \beta + \sum_{(i,j) \ne (1,1)} x_{i,j} t_{i,j} \ge 0,\]
for every positive linear functional $(\beta, T_0).$

We claim that $(\mu, X_0)$ is positive if and only if $\mu I +X_0
\in \cl B(\ell^2(\bb N))$ is a positive operator. This will show
that the bijection $$(\mu,X_0) \in \cl S_0^{**} \mapsto \mu I +X_0
\in \cl B(\ell^2(\bb N))$$ is an order isomorphism. Also, note
that the identity of $\cl S_0^{**}$ is $(1,0),$ so that this map
is unital.

To see the claim, first let $(\mu, X_0) \in \cl S_0^{**}$ be
positive. Given any $T= T_0 + \alpha E_{1,1}$ a positive trace
class operator, let $\beta = \alpha + {\rm tr}(T_0) = {\rm
tr}(T).$ Then $(\beta, T_0)$ is positive in $\cl S_0^*$ and, hence
\[ 0 \le \mu \beta + \sum_{(i,j) \ne (1,1)} x_{i,j} t_{i,j} =
\mu {\rm tr}(T) + {\rm tr}(X_0^t T) = {\rm tr}((\mu I+ X_0)^t T).
\] Since $T$ was an arbitrary trace class operator, this shows
that $\mu I +X_0$ is a positive operator in $\cl B(\ell^2(\bb
N)).$

Conversely, if $\mu I + X_0$ is a positive operator, then for any
positive $(\beta, T_0) \in \cl S_0^*$, pick $\alpha$ as above and
set $T = \alpha E_{1,1} + T_0$. We have that
\[ \mu \beta + \sum_{(i,j) \ne (1,1)} x_{i,j} t_{i,j} \ge \mu {\rm tr}(T) + {\rm tr}(X_0^t T) = {\rm tr}((\mu I + X_0)^t T) \ge 0, \]
since both operators are positive.

This completes the proof of the claim and of the theorem.
\end{proof}


\end{document}